\documentclass{amsart}

\usepackage{amsmath}
\usepackage{amssymb}

\DeclareMathSymbol{\R}{\mathalpha}{AMSb}{"52}
\DeclareMathSymbol{\C}{\mathalpha}{AMSb}{"43}
\DeclareMathSymbol{\D}{\mathalpha}{AMSb}{"44}

\DeclareMathOperator{\re}{Re}
\DeclareMathOperator{\im}{Im}
\newcommand{\SHO}{S_{\scriptscriptstyle H}^{\scriptscriptstyle O}}
\newcommand{\KHO}{K_{\scriptscriptstyle H}^{\scriptscriptstyle O}}
\newcommand{\CHO}{C_{\scriptscriptstyle H}^{\scriptscriptstyle O}}

\newcommand{\THHO}{T_{\scriptscriptstyle H}^{\scriptscriptstyle O}}
\newcommand{\TRHH}{T_{\scriptscriptstyle H}}
\newcommand{\TRH}{{\mathcal T}_{\scriptscriptstyle H}}
\newcommand{\THO}{{\mathcal T}_{\scriptscriptstyle H}^{\scriptscriptstyle O}}

\newcommand{\ol}{\overline}
\newcommand{\g}{\overline{g}}
\newcommand{\z}{\overline{z}}
\newcommand{\p}{\partial}

\theoremstyle{plain}
\newtheorem{thm}{Theorem}

\newtheorem*{cor}{Corollary}

\theoremstyle{definition}
\newtheorem{dfn}{Definition}

\newtheorem*{pro}{Proposition}
\newtheorem*{dfnA}{Definition [Clunie and Sheil-Small]}
\newtheorem{rmk}{Remark}
\newtheorem*{exam}{Example}

\begin{document}

\title[]{Typically Real Harmonic Functions}
\date{January 31, 2008}

\author[Dorff]{Michael Dorff}
\email{mdorff@math.byu.edu}
\address{Department of Mathematics, Brigham Young
University, Provo, Utah 84604, USA}

\author[Nowak]{Maria Nowak}
\email{nowakm@golem.umcs.lublin.pl}
\address{Instytut Matematyki, UMCS, pl. Marii Curie-Sk\l odowskiej
1, 20-031 Lublin, Poland}

\author[Szapiel]{Wojciech Szapiel}
\email{szawoj@kul.lublin.pl}
\address{Instytut Matematyki i Informatyki, KUL, ul. Konstantyn\'ow 1
H, 20-708 Lublin, Poland}

\keywords{Harmonic Mappings, Typically Real, Univalence}
\subjclass{30C45}

\begin{abstract}
We consider a class $\THO$ of typically real harmonic
functions on the unit disk that contains the class of  normalized
analytic and  typically real functions. We also obtain  some
partial results about the region of univalence for this class.
\end{abstract}

\maketitle

\section{Introduction}



A planar harmonic mapping is a complex-valued function $f=u+iv$, for which both $u$
and $v$ are real harmonic. If $G$ is simply connected, then $f$
can be written as $f=h+\overline g$, where $h$ and $g$ are
analytic  on $G$. The reader is referred to \cite{dur} for many
interesting results on planar harmonic mappings. Throughout this paper we
will discuss harmonic functions on the unit disk $\mathbb D$.  In analogue to the
classical family $S$ of normalized analytic schlicht functions and its subfamilies $K$ 
of convex mappings and $C$ of close-to-convex mappings, 
Clunie and Sheil-Small \cite{css} introduced the class $\SHO=
\{f : \D \rightarrow {\mathbb C}~\big|~f \mbox{ is harmonic, 
univalent with } f(0)=h(0)=0, f_{z}(0)=h'(0)=1, f_{\overline{z}}(0)=g'(0)=0 \}$ 
and its corresponding subclasses $\KHO$ and $\CHO$. Note that $S \subset \SHO$, 
$K \subset \KHO$, and $C \subset \CHO$. Another well-known class of analytic 
functions in $\D$ is the family, $T$, of typically real functions that 
have the normalization $f(z)=z+a_2z^2+ \cdots$ and are real 
if and only if $z$ is real. Clunie and Sheil-Small introduced 
the family of harmonic typically real functions $f $ for
which $f(z)$ is real if and only if $z$ is real. Then they proposed the 
following class of harmonic typically real functions.
%

\begin{dfnA}
Let $\TRHH$ be the class of typically real harmonic functions
$f=h+\g$ such that $|g'(z)|<|h'(z)|$ for all $z \in \D$, $f(0)=0$,
$|h'(0)|=1$, and $f(r)>0$  for $0<r<1$. Let $\THHO$ be the
subclass of $\TRHH$ with $g'(0)=0$.
\end{dfnA}

Note that $\TRHH$ is normal and $\THHO$ is compact. 
Besides Clunie and Sheil-Small, several other authors have investigated 
harmonic real real functions (see \cite{bhh}, \cite{wlz}). 

The condition that $|h'(z)|>|g'(z)|$ means that $f=h+\bar g$ must be 
locally univalent and sense--preserving (see  Lewy \cite{lew}). However, 
not all analytic typically real functions are locally univalent. Thus, a 
problem with this definition is that it prevents the family of analytic 
typically real functions from being a subset of their family of harmonic 
typically real functions. That is, $T \not\subset \THHO$. 


To resolve this problem and allow all analytic
typically real functions to be also harmonic typically real
functions, we offer a slightly different definition for a family of 
harmonic typically real functions, $\THO$. 
In particular, we reduce the requirement that the harmonic 
functions must be locally univalent. This means that the standard results 
for harmonic locally univalent functions must be reconsidered for this family. 
We therefore show that for the family $\THO$ Clunie and Sheil-Small's 
shearing technique still holds. Also, as in the case for the family of analytic typically real functions
we investigate the region of univalency for the harmonic family and provide several conjectures
for $\THO$.




\section {The class $\THO$}

For the harmonic function $f=h+\g$, let  $\omega$ be given by
$g'(z)=\omega (z)h'(z)$. We say that $f$ is sense-preserving at a point $z_0$ if 
$h' (z) \not\equiv 0$ in some neighborhood of $z_0$ and $\omega$ is analytic at $z_0$ with $|\omega(z_0)|<1$. If $f$ is 
sense-preserving at $z_0$, then either the Jacobian 
$J_f(z_0)=|h'(z_0)|^2-|g'(z_0)|^2>0$ or $h'(z_0)=0$ for an isolated point 
$z_0$ as was mentioned by Duren, Hengartner, and Laugesen \cite{dhl}.
That is, $z_0$ is a removable singularity of the meromorphic function 
$\omega$ and $|\omega (z_0)|<1$.
We say $f$ is sense-preserving in $\D$ if $f$ is sense-preserving at all $z \in \D$.
By requiring the harmonic function $f$ to be sense-preserving we retain
some important properties exhibited by analytic functions, such as the open mapping 
property, the argument principle, and zeros being isolated (see
\cite{dhl}). We note that the following harmonic typically real functions 
\begin{equation*}
 f_1(z)=z-\bar z \hspace{.2in} \mbox{and} \hspace{.2in} 
f_2(z)=2(1+i)z+iz^2+\ol{2(-1+i)z+iz^2}.
\end{equation*}
are not sense-preserving, and they do not have the properties mentioned above.

Thus, we give 
the following definition.

\begin{dfn}
Let $\TRH$ be the class of typically real harmonic functions, $f$, 
such that $f$ is a sense-preserving harmonic function,  $f(z)$ 
is real if and only if $z$ is real, $f(0)=0$, $|h'(0)|=1$, and $f(r)>0
\;\text{for } 0<r<1$. Let $\THO$ be the subclass of $\TRH$ with
$g'(0)=0$.
\end{dfn}

Also, notice that $T \cup\THHO\subset \THO$, and with this definition, 
as in the analytic case, a harmonic typically real function need not be univalent or
even locally univalent.

\begin{thm}
If $f \in \TRH$, then $f$ is strictly increasing on the real
interval $(-1,1).$ Moreover, if $f=h+\g\in\THO$, then
$h'(0)=f_{z}(0)=1$.
\end{thm}

\begin{proof} Observe that the derivative  $f'$ exists on  the interval $(-1,1)$ and
$f'=h'+\overline{g'}$,  $\im h=\im g$ there. Suppose that there
exists a point $x_0 \in (-1,1)$ such that $f'(x_0)=0$. This
implies that $J_f(x_0)=0$. As we know this can only occur
if $h'(x_0)=0=g'(x_0)$ with the order of the zero of $g'$ greater
than or equal to the order of $h'$. Hence, $(h-g)'(x_0)=0$ contrary to the fact 
that $h-g$ is a typically real analytic function and such functions are known to be 
univalent in the lens domain bounded by the
circles $|z\pm i|=\sqrt 2$ (\cite{gol},\cite{mer}).
\end{proof}

Now, we note that the basic shearing theorem by Clunie and Sheil-Small [\cite{css}, Theorem
 5.3] still holds when local univalence is omitted. That is, we have the following version.

\begin{thm}
Let $f=h+\g$ be sense-preserving harmonic on $\D$.
Then $f$ is univalent and convex in the horizontal direction on
$\D$ if and only if $h-g$ has the same properties.
\end{thm}
\begin{proof}
We only need to prove the reverse direction. So assume that
$F=h-g$ is univalent and convex in the horizontal direction.
Consider
\begin{equation*}
G(w)=f(F^{-1}(w))= h(F^{-1}(w))+\ol{g(F^{-1}(w))}=w+2\re \big\{ g(F^{-1}(w)) \big\}.
\end{equation*}
If $G$ is locally univalent in $\Omega=F(\D)$, then we can apply
the same approach as in Clunie and Sheil-Small's proof. In
particular, by their lemma  ( \cite{css}, p. 13), $G$ is univalent
in $\Omega$ and has an image that is convex in the horizontal
direction, and consequently, so is $f$. Therefore, we only need to
show that $G$ is locally univalent. To do this, consider the
Jacobian of $G$:
\begin{align*}
J_G= &   \bigg|  \dfrac{d }{d w} h\circ F^{-1} \bigg|^2-\bigg|
        \dfrac{ d}{d w} g\circ {F}^{-1}\bigg|^2 \\
        = & (|h'\circ  F^{-1}|^2-|g'\circ F^{-1}|^2)\cdot
        |(F^{-1})'|^2=
         J_{f\circ F^{-1}}\cdot |(F^{-1})'|^2 .
\end{align*}
Now suppose there exists a point $z_0 \in \D$ such that
$J_G(F(z_0))=0$. Since $ (F^{-1})'(w) \neq 0$ on $F(\D)$, we have
that $|h'(z_0)|=|g'(z_0)|$. As mentioned above, this is only
possible when $h'(z_0)=0=g'(z_0)$ which contradicts the assumption
that $F=h-g$ is univalent.
\end{proof}

Next, we give a representation formula and extreme points 
for functions in the class $\THO$.

Let $\mathcal P$ denote the class of all functions of the form
$p(z)=1+p_1z+p_2z^2+\dots $ that are analytic in $\mathbb D$
 and such that $\re p(z)>0$  for $z\in \mathbb D$.
By the well-known Herglotz representation formula $p\in\mathcal P$
if and only if there exists  a unique probability measure $\mu$ on
 $\p \D$  such that
 \begin{equation*}p(z)=\int_{\p\D}\bold p_{\eta}(z)d\mu(\eta),\quad z\in\mathbb
 D,\tag{2.1}\end{equation*} where
\begin{equation*} \bold p_{\eta}(z) =(1+\eta z)/(1-\eta z). \tag{2.2}\end{equation*}
Moreover, if $p\in\mathcal P$ has real Taylor coefficients, then
$$p(z)=\int_{-1}^{1}\frac {1-z^2}{1-2tz+z^2}d\nu(t),\quad z\in\D$$
with the unique probability measure $\nu$ on the segment $[-1,1]$.
This in turn  implies that
 for an analytic function $ F$ in the class $T$ we have the
following Robertson representation formula
\begin{equation*} F(z)=\int_{-1}^{1}\frac {zd\nu(t)}{1-2tz+z^2},
\quad z\in \D, \tag{2.3}\end{equation*} where $\nu$ is as above.
The set of extreme points of the class $T$ consists of the
functions
\begin{equation*} z\mapsto \bold q_t(z)=\frac{z}{1-2tz+z^2},\quad  -1\leq t\leq 1.
\tag{2.4}\end{equation*}

 The shearing construction can
be applied to the class $\THO$. Consequently, we see that every
$f=h+\bar g\in \THO$ can be written in the form
\begin{equation*}
 f(z)=\text{Re}
\int _0^zp(\zeta)F'(\zeta)d\zeta+ i\ \text{Im}F (z)=k(z,
p,F),\tag{2.5}\end{equation*} where $F = h-g\in T$ and
$p=(1+\omega)/(1-\omega)\in \mathcal P$ with $ \omega=g'/h'$, where
 removeable singularities are admitted. Also,
given $p\in\mathcal P $ and $F\in T,$ the function $f$ defined by
(2.5) is in $\THO$ and
 $k(\cdot, p,F)=h+\bar g$, with
\begin{align*}
h(z)&=\frac 12\int_0^z(p(\zeta)+1)F'(\zeta)d\zeta=
z+a_2z^2+\dots,\\ g(z)&=\frac
12\int_0^z(p(\zeta)-1)F'(\zeta)d\zeta =b_2z^2+b_3z^3+\dots.
\end{align*}
Note also that the function $f=k(\cdot, p,F)$ is locally univalent
if and only if $F$ is a locally univalent function. This is a
consequence of the equality \begin{equation*}J_f(z)=|F'(z)|^2\re
p(z),\quad z\in \mathbb D.\tag{2.6}\end{equation*}
 Furthermore   we have
\begin{thm}
The class $\THO$ is compact (in the topology of uniform
convergence on the compact subsets of $\mathbb D$) and the set
$\text{ext} (\THO) $ of its extreme points consists of the
functions $k(\cdot, \bold p_{\eta},\bold q_t)$, where $\bold
p_{\eta}$ and $\bold q_t$
  are given by (2.2) and (2.4), respectively. The class
$\THO$ is not convex.
\end{thm}
\begin{proof}
Compactness of the class $\THO$ follows immediately from
compactness of both classes $T$ and $\mathcal P$. Assume that
$f=k(\cdot,p,F) \in\text{ext} (\THO) $  and  there is
$0<\lambda<1$ such that either
\begin{equation*}
p=(1-\lambda) p_1+\lambda p_2,\ \ \text{with}\  p_1, p_2\in
\mathcal P, \  p_1\ne  p_2, \tag{i}
\end{equation*}
 or
\begin{equation*}
F=(1-\lambda)F_1+\lambda F_2,\ \ \text{with}\ F_1,F_2\in T, \
F_1\ne F_2. \tag{ii}
\end{equation*}
Then $$ f=(1-\lambda)f_1+\lambda f_2, $$ where, in case (i): $$
f_j=k(\cdot, p_j,F)\quad\text{with}\quad (f_1)_z-(f_2)_z=(p_1-
p_2)F'/2,$$ which implies $f_1\ne f_2$, a contradiction; and in
case (ii): $$ f_j=k(\cdot, p,F_j)\quad\text{with}\quad
(f_1)_z-(f_2)_z=(p+1)(F_1'-F_2')/2,$$
 a contradiction again. Thus, by the Herglotz and Robertson
 formulas, we get $\text{ext}(\THO)\subset \{k(\cdot,
\bold p_{\eta},\bold q_t), |\eta|=1, -1\leq t\leq 1\}$. Now if  $$
f=k(\cdot, \bold p_{\eta},\bold q_t)= (1-\lambda)f_1+\lambda f_2=
(1-\lambda)k(\cdot, p_{1},F_1) +\lambda k(\cdot,  p_{2},F_2),$$
then $$\bold q_t'=f_z-\overline{f_{\bar z}}=(1-\lambda)F_1'+
\lambda F_2',$$ which gives $\bold q_t=F_1=F_2$; and $$\bold
p_{\eta}\bold q_t'=f_z+\overline{f_{\bar z}}= (1-\lambda)
p_1F_1'+\lambda
 p_2F_2'= ((1-\lambda)p_1+\lambda p_2)\bold q_t'   ,$$ which implies $ p_1=p_2=\bold p_{\eta}$.
Consequently, $ f_1=f_2$ and $f\in \text{ext}(\THO).$

Finally, we  show that the class $\THO$ is not convex. More
exactly, we  show that for arbitrary $\xi,\eta\in \p \mathbb D,
s,t\in[-1,1], \xi\ne\eta,s\ne t$ and $0<\lambda<1 $, $$f=
(1-\lambda)k(\cdot,\bold p_{\xi}, \bold q_s)+\lambda k(\cdot,\bold
p_{\eta}, \bold q_t) \notin \THO.$$ Suppose, contrary to our
claim, that $f\in \THO$. Then there exist $p\in\mathcal P$ and
$F\in T$ such that $f=k(\cdot,p, F)$ and
$$F'=f_z-\overline{f_{\bar z}}=(1-\lambda)\bold q_s'+\lambda \bold
q_t'.$$ This implies that $F=(1-\lambda)\bold q_s+\lambda\bold
q_t.$ Moreover, we have $$pF'=f_z+\overline{f_{\bar
z}}=(1-\lambda)\bold p_{\xi}\bold q_s'+\lambda \bold p_{\eta}
\bold q_t'.$$ Since the image of $\mathbb D$ under an analytic
branch of $\sqrt{\bold q_s'/\bold q_t'}$ contains the upper and
lower half planes, there exists  an $a\in\mathbb D\setminus\{0\}$
such that $\bold q_s'(a)/\bold q_t'(a)=-\lambda/(1-\lambda).$
Hence $F'(a)=0$ and $$p(a)F'(a)=(1-\lambda)\bold p_{\xi}(a)\bold
q_s'(a)+\lambda\bold  p_{\eta}(a)\bold q_t'(a)=\lambda \bold
q_t'(a)(\bold p_{\eta}(a)-\bold p_{\xi}(a))\ne0,$$ a
contradiction.
\end{proof}

As a corollary to Theorem 3 we get the same sharp coefficient
estimates for the class $\TRH$ and $\THO$ as were found by Clunie
and Sheil-Small \cite{css}  for $\TRHH\subset \TRH$ and
$\THHO\subset\THO$.

\section{Region  of univalence}

For $z_0\in \mathbb C$ and positive $r$ let   $D(z_0; r)$ denote
the open disk centered at $z_0$ with the radius $r$. We have
mentioned in the Introduction that an analytic function $f\in T$
need not  be univalent in $\mathbb D$, but it is univalent in the
lens domain $$L=D(-i;\sqrt 2)\cap D(i;\sqrt 2).$$
 The result was obtained
by  Goluzin \cite{gol} and  by Merkes \cite{mer} independently.
They also noted that  this region of univalence for the class $T$
cannot be extended, because for each $z_0\in \p L\cap \D $ there
exists a parameter $t_0 \in (0,1)$ such that $f_{t_0}'(z_0)=0$,
where
\begin{equation*}
\label{eg:extfnct}
f_t(z)=\dfrac{t z}{(1-z)^2}+\dfrac{(1-t)z}{(1+z)^2}.\tag{3.1}
\end{equation*}
This can be also showed by noting that $$\p L\cap
\D=\left\{z\in\D: \left(\frac {1+z}{1-z}\right)^4<0\right\}$$ and
$$ f_t'(z)=\left(\left(\frac{1+z}{1-z}\right)^4+\frac
{1-t}t\right)\frac {t(1-z)}{(1+z)^3}.$$ Let us observe  that
actually for each $z_0\in\D\setminus L$ there exist $t_0\in (0,1)$
and $R \in (\sqrt 2-1,1]$ such that $Rz_0\in\p L$ and
$f_{t_0,R}'(z_0)=0$, where $f_{t,R}(z)= f_t(R z)/R$ and $f_t$ is
defined by (3.1). Note that the function  $f_{t,R}$ as a convex
combination of univalent functions with real coefficients is in
the class $T$.

 As in the analytic case, a harmonic typically real function need not
be univalent. Therefore, E. Z\l otkiewicz posed the problem of
determining the region of univalence for harmonic typically real
functions. Before we give a partial answer to this question we
present a simple proof of the Goluzin-Merkes result for analytic
typically real functions (based on Merkes' idea). To this end note
first that the function \begin{equation*} \zeta=\psi (z)=
\frac{2z}{1+z^2}\tag{3.2}\end{equation*} maps conformally the disk
$\mathbb D$ onto the two-slit plane cut along the real intervals
$(-\infty,-1]$ and $[1,\infty)$. Since the function $\psi$  is
typically real,  there is a one-to-one correspondence between the
class $T$ and the class of normalized and typically real functions
in $\Omega= \mathbb C\setminus ((-\infty, -1]\cup[1,\infty))$.
Moreover, using the Robertson formula we get the following formula
for a typically real function $F$  in $\Omega$ with normalization
$F(0)=F'(0)-1=0$ and the  one-to-one correspondence:
\begin{equation*} F(\zeta)= \int_{-1}^1\frac {\zeta
d\nu(t)}{1-t\zeta}\ ,\quad f=\frac 12 F\circ \psi\in T,
\tag{3.3}\end{equation*} where $\nu$ is a probability measure on
$[-1,1]$. It has been observed in \cite{thal} and \cite{tcha} that
$F$ restricted to the disk $\mathbb D$ is univalent. Consequently,
any function $f\in T$ is univalent on the preimage of the unit
disk under the function $\psi$ given by (3.2), which is the lens
domain $L$.

In 1936 Robertson observed that an analytic function $F$ with real
coefficients is univalent and convex in the vertical direction if
and only if the function $z\mapsto z F'(z)$ is typically real (see
\cite{go1}, p. 206). Hence the functions given by (3.3) are convex in the
direction of the imaginary axis  (see also \cite{reto},
\cite{mer}). Therefore the sets $f(L), f\in T,$  are convex in the
vertical direction. Moreover, we will show the following
interesting property of the  class $T$.
\begin{pro} For a $z\in\p L\cap\D$ there exists a
unique $f\in T$ for which  $f'(z)=0$.
\end{pro}
\begin{proof} By (3.3) it is enough to consider the equation
\begin{align*} 0=F'(e^{i\alpha})=& \int_{1}^1\frac {d\nu(t)}{(1-te^{i\alpha})^2}\\
= & \int_{-1}^1\frac{1-t^2}{|1-te^{i\alpha}|^4}d\nu(t) -2\cos\alpha
\int_{-1}^1\frac{t(1-t\cos \alpha)}{|1-te^{i\alpha}|^4}d\nu(t) \\
& \hspace{1in} +2i\sin\alpha\int_{-1}^1\frac{t(1-t\cos
\alpha)}{|1-te^{i\alpha}|^4}d\nu(t),\end{align*} where
$0<\alpha<\pi.$ It then follows
\begin{equation*}\int_{-1}^1\frac{t(1-t\cos
\alpha)}{|1-te^{i\alpha}|^4}d\nu(t)=0 \tag {i} \end{equation*} and
consequently,
\begin{equation*}\int_{-1}^1\frac{1-t^2}{|1-te^{i\alpha}|^4}d\nu(t)=0.\tag {ii}\end{equation*} From
 equality (ii) we get
$\nu=(1-\lambda)\delta_{-1}+\lambda\delta_1$ for some $\lambda\in
[0,1].$ Finally, equality (i) gives $\lambda=\sin^2(\alpha/2).$
\end{proof}
\begin{cor} Let $f\in T$. Then either $f$ is univalent on $\ol
L\setminus \{-1,1\}$ or there is a unique $t\in (0,1)$ such that
$f=f_t$, where $f_t$ is given by (3.1). Moreover, $f_t(L)=\mathbb
C\setminus \{(1-2t)/4+i\lambda:\lambda \in \mathbb R,
|\lambda|\geq \sqrt{t(1-t)}/2\}.$\end{cor}
\begin{proof}
  Clearly $f$ is analytic on $\gamma= \p
L\setminus\{-1,1\} $ and $\re f(z)$  changes monotonically. It is
sufficient to show  that $\re f(z)$ is not constant on any arc
$\gamma_0\subset \gamma$ or $f=f_t$ for some $t\in (0,1).$ If $f$
is constant on an arc  $\gamma_0\subset \gamma$ lying  in the
upper half-plane , then the function given by
$$ g(z)= f(z)+ \overline { f\left(-i+\frac 2{\bar z-i}\right)}$$
 is  analytic on a neighborhood of $\gamma_0$ and $ g(z)=2\re f(z)$ on $\gamma_0$.
So, $g(z)=const$ on $\gamma_0$ and consequently, $g$ is a constant
function. This  means that $\re f$ is constant on $\gamma$.
Consequently, the boundary value of $f$ at $1$ and $-1$ is equal
to $\infty$,  so there is  $z\in \p L\cap\D$ such that
$f'(z)=0=F'(\psi(z))=0$. Hence by Proposition $f=f_t$, where
$t=(1-\re \psi(z))/2.$\end{proof}

We also note that the radius of starlikeness for the class $T$ is
$\sqrt 2-1$ \cite{kir}. Moreover, every $f\in T$ is univalent on
$\ol {D(0;\sqrt 2-1)}$ and the curve $f(\p D(0;\sqrt 2-1)) $ is
strictly starlike with respect to the origin. Indeed, if we put
$g=zf'/f$, then  the function defined by $  G(z)=
g(z)+\ol{g((3-2\sqrt 2)/\z)} $ is analytic on a neighborhood of
the circle $\p D(0;\sqrt 2-1)$. Hence for $|z|=\sqrt 2-1$ we have
$G(z)=2\re \{zf'(z)/f(z)\}>0$, except for a finite number of
points at which it vanishes.

We have already showed that every harmonic typically real function
in the sense of Definition 1 is strictly monotonic on the interval
$(-1,1)$.  Moreover, we have the following

\begin{thm}
\label{thm4}
For each function $f$ in $\THO$  there exists an open set $V$, $
(-1,1)\subset V\subset \D $,  such that $f$ is univalent on $V$.
\end{thm}

\begin{proof}
Let $f=k(\cdot, p, F)$ with $p\in\mathcal P$ and $ F\in T$.  We
first show that for a compact interval $[a,b]\subset (-1,1)$ there
is an open set $U$ containing $[a,b]$ and such that $f$ is
univalent on $U$. Clearly,
 $[F(a),F(b)]\subset F(L)$, where $L$ is the lens domain defined above.
 Since F(L) is an open set, there exist $\delta>0$ and $c>0$
  such that $(F(a)-\delta,
F(b)+\delta)\times (-c,c)\subset F(L)$. Let $U$ be the preimage of
the set $(F(a)-\delta, F(b)+\delta)\times (-c,c)$ under $F$. Then
$$U= U(a,b,c,\delta)=\bigcup_{-c<d<c}z_d((F(a)-\delta,
F(b)+\delta)),$$ where $z_d(t)=F^{-1}(t+id), \
F(a)-\delta<t<F(b)+\delta.$  Now note that since $F$ is univalent
on $L$, the curves $z_d, -c<d<c,$ are disjoint and
$$\frac{d}{dt}\re f(z_d(t))=\re \{p(z_d(t))F'(z_d(t))z'_d(t)\}=\re
p(z_d(t))>0.$$ This and the fact that $\im f=\im F$  imply the
univalence of $f$ on $U$.

 Let $\{a_n\}$ be a strictly decreasing
sequence of negative numbers converging to -1  and $\{b_n\}$ be a
strictly increasing sequence of positive numbers converging to 1.
Then for each positive integer $n$,  we can find $\delta_n>0$,
$c_n>0$ and the open set $U_n=U(a_n,b_n, c_n,\delta_n)$ such that
$f$ is univalent on $U_n$. Now set
$\delta_n'=\min\{F(a_n)-F(a_{n+1}), F(b_{n+1})-F(b_n), \delta_n\}$
and $c_1'=c_1, c_{n+1}'=\min \{c_n', c_{n+1}\}, n=1,2,\dots ,$ and
define $$V=\bigcup_{n=1}^{\infty} U(a_n, b_n, c_n', \delta_n').$$
Clearly, $(-1,1)\subset V.$ Moreover, $f$ is univalent on $V$. To
see this suppose that  $f(z)=f(w)$ and $z\in U(a_n, b_n, c_n',
\delta_n')$, $w\in U(a_{n+k}, b_{n+k}, c_{n+k}',\delta_{n+k}'),
k\geq 1 $. Since $\im F=\im f$,  we get $z\in U(a_{n+k}, b_{n+k},
c_{n+k}', \delta_{n+k}')$ and consequently, $z=w.$
\end{proof}

\begin{rmk}
It is clear that for every continuous mapping $f$ of a neighborhood 
of the interval $(-1,1)$ into $\C$ such that $f((-1,1)) \subset \R$ and 
$f$ is a local homeomorphism of $(-1,1)$, there is a domain $\Omega$
and a simply connected domain $G$ such that $(-1,1) \subset \Omega$
and $f$ is a local homeomorphism of $\Omega$ onto $G$. If the pair 
$(\Omega,f)$ is an unlimited covering space of the domain $G$, then by
the Monodromy Theorem $f$ is a homeomorphism of $\Omega$ onto $G$
\cite{as}. In general, such a situation is rare. The example below shows 
that $f$ may be infinite-valent on $\Omega$, so that the typically real property
in the proof of Theorem \ref{thm4} seems to be essential.
\begin{exam}
Let $u(z) \equiv \frac{4z}{(1+z)^2}$, $f( \xi) \equiv \xi e^{-\xi}$. It is clear 
that the function $f \circ u$ is locally univalent on $\D$. By the Great Picard 
Theorem, $f \circ u (\D) = \C$ and every value $w \in \C \setminus \{0\}$
is assumed by $f \circ u$ at infinitely many points of each set
$\D \cap \{ z : |z+1| < \delta \}$, where $0 < \delta < 2$.
\end{exam}
\end{rmk}

Next, we show that the region, $L$, of univalency for the class $T$  
is not the region of univalency for the class $\THO$.

\begin{thm} There exist functions $f\in \THO$ that
are not univalent on $L$.
\end{thm}
\begin{proof}
Put $$ F(z)=f_{1/2}(z)=\frac 12\left(\frac z{(1+z)^2}+\frac
z{(1-z)^2}\right),\quad z\in \mathbb D,$$ and define $f\in \THO$
by the formula $$ f(z)=\re
\int_0^z\frac{1+\zeta}{1-\zeta}F'(\zeta)d\zeta+i\im F(z).$$
Suppose that $f$ is univalent on $L.$ Then the function $g=f\circ
\psi^{-1},$ where $\psi$ is given by (3.2), is univalent on
$\mathbb D$. A calculation gives $$ g(w)=\re\left(\frac
{1+w}{12(1-w)}\sqrt{\frac{1+w}{1-w}}-\frac
14\sqrt{\frac{1-w}{1+w}}+\frac 16\right)+\frac{i}2\im \left(\frac
w{1-w^2}\right),$$ where we assume that $\sqrt 1=1$.
 Now, note that for $0<\alpha<\pi/2$,
$$\im\left( g(ie^{-i\alpha})-g(ie^{i\alpha})\right)=0.$$ Moreover,
we have
\begin{align*} \re &\left(
g(ie^{-i\alpha})-g(ie^{i\alpha})\right)\\ & =\frac
1{12\sqrt2}\left(\cot^{3/2}(\frac{\pi}4+\frac{\alpha}2)-\cot^{3/2}(\frac{\pi}4-\frac{\alpha}2)\right)-
\frac
1{4\sqrt2}\left(\cot^{1/2}(\frac{\pi}4+\frac{\alpha}2)-\cot^{1/2}(\frac{\pi}4-\frac{\alpha}2)\right)\\
&  =\frac
C{12\sqrt2}\left(\cot^{1/2}(\frac{\pi}4+\frac{\alpha}2)-\cot^{1/2}(\frac{\pi}4-\frac{\alpha}2)\right),
\end{align*}
where $$C=\cot(\frac{\pi}4+\frac{\alpha}2)- 2 + \frac
1{\cot(\frac{\pi}4+\frac{\alpha}2)}>0.$$ This means that for
$0<\alpha<\pi/2$, $$ \re \left(
g(ie^{-i\alpha})-g(ie^{i\alpha})\right)<0.$$ To get a
contradiction consider the function $m$ defined by
  $$m(r,\alpha)= \re \left(
  g(rie^{-i\alpha})-g(rie^{i\alpha})\right).$$
  The function $m$ is uniformly continuous on the rectangle
  $[0,1]\times [0, \pi/4]$ and $m(1,\alpha)<0$ for $0<\alpha<\pi/4.$
  On the other hand,
  $$ m(r,\alpha)=r(\sin\alpha+o(1))\quad \text{as}\ r\to 0^+.$$
  Consequently, for every $\alpha\in (0,\pi/4)$ there is
  $r_{\alpha}\in(0,1)$ such that $ m(r_{\alpha},\alpha)=0$. This
  means that
  $g(r_{\alpha}ie^{-i\alpha})=g(r_{\alpha}ie^{i\alpha})$, a contradiction.

\end{proof}

\begin{thm}
Every function $f\in \THO$ is univalent in any of the following
domains
\item [(a)]
 the disk $D(0; \sqrt6-\sqrt5)$,
\item [(b)] $\left\{z\in \D: \left|\frac{2z}{1+z^2}\right|<\sqrt
2-1\right\}.$

\end{thm}
\begin{proof} It follows from (2.6) that  every $f=h+\bar g\in \THO$ is locally
univalent on the lens domain L. Moreover,  by the results in
\cite{tod} $,   F= h-g $ is convex on $D(0;\sqrt 6-\sqrt 5)$.
Thus, by the shearing   theorem of Clunie and Sheil-Small $f$ is
univalent on $D(0;\sqrt 6-\sqrt 5)$. Note also that it has been
showed by Koczan \cite {koc} that for the class $T$ the radius of
convexity in the horizontal direction is exactly $\sqrt 6-\sqrt
5$. Now  we observe that a   function $f\in \THO$ is univalent  on
the given region in (b)  if and only if function $f\circ\psi$,
where $\psi$ is given by (3.2) is univalent on the disk $D(0;
\sqrt 2-1). $ The last follows from the fact that an analytic
function $F$ given by (3.3) maps the disk $D(0;\sqrt 2-1)$ onto a
convex domain (see \cite{reto}, p. 292] ) and from the shearing
theorem of Clunie and Sheil-Small.
\end{proof}

Clearly,  the class $T_H$ of typically real harmonic functions
introduced by Clunie and Sheil-Small contains locally univalent   functions from the
class $T$. It would be interesting to find the region of
univalence for  locally univalent functions that are in $T$. The
following   example of  the function $G\in T$ that is locally
univalent  has been described  in \cite{goo}:
 $$G(z)=\frac 1{\pi}\tan\left(\frac {\pi z}{1+z^2}\right),\quad z\in\D.$$
We note  that $G$ is univalent in the region $S=\left\{z\in \D:
\left|\re \frac{\pi z}{1+z^2}\right|<\frac {\pi}2\right\}$ which
contains the  disk $D(0;1/\sqrt 3)$. Indeed, for $|z|=1/\sqrt 3$,
we have $$\left|\re \frac{\pi z}{1+z^2}\right|=
\left|\frac{3\pi\re z}{9\re ^2z+1}\right|\leq\frac {\pi}2.$$
 Moreover, if $z_0=(1+i\sqrt 2)/3$, then
$z_0,-\overline{z_0 }\in\p D(0;1/\sqrt 3)\cap \p S $ and
$G(z_0)=G(- \overline{z_0})$. This shows that radius of univalence
for the class of locally univalent functions from $T$ is less than
or equal to $1/\sqrt 3$.

Now let $r^*_u$ (resp. $r_u$) denote the radius of univalence of
$T_H$ (resp. $\THO$), that is the supremum of all $r>0$ such that
every $f\in T_H$ ( resp. $f\in\THO$) is univalent on $D(0;r)$.
Clearly, $$0.213...=\sqrt 6-\sqrt 5\leq r_u\leq r^*_u\leq 1/\sqrt
3=0.577...$$ and $$r_u\leq \sqrt 2-1=0.414\dots\ .$$

By examining some computer computations, that will be presented in an upcoming paper, we make the following conjectures.

\vskip10pt {\bf Conjecture 1.} $r_u=\sqrt 2-1$.

{\bf Conjecture 2.} Every function $f\in\THO$ is univalent on the
half-lens
$$ L\cap\{z: \re z> 0\}.$$

\vskip10pt

We finish the paper with the list of open problems.

\begin{enumerate} \item Give analytic proofs of Conjectures 1-2.
\item
 Prove or disprove that $r_u^*=1/\sqrt 3.$
\item  Does exist an open set $U$, $ (-1,1)\subset U\subset\D$,
such that every $f\in\THO$ is univalent on $U$?
\end{enumerate}
\vskip1cm

\end{document}